\documentclass[a4paper,12pt]{amsart}

\usepackage{amsmath}
\usepackage{amssymb}
\usepackage{amsfonts}

\newtheorem{thm}{Theorem}[section]
\newtheorem{prop}[thm]{Proposition}
\newtheorem{lem}[thm]{Lemma}
\newtheorem{cor}[thm]{Corollary}

\newtheorem{defi}[thm]{Definition}
\newtheorem{exo}{\bf\large Exercice}

\newcommand{\R}{\mathbb{R}}

\def\vec#1{\boldsymbol{#1}}
\newcommand{\I}{\infty}
\newcommand{\as}{\mathbf{A1}}

\newcommand{\dd}{\mathbf{D}}
\newcommand{\bb}{\mathbf{B}}
\newcommand{\bo}{\vec{\omega}}
\newcommand{\bs}{\mathbf{S}}
\newcommand{\II}{\mathbf{I}}
\newcommand{\T}{\mathbf{T}}
\newcommand{\vv}{\mathbf{v}}
\def\bpsi{\vec{\psi}}
\newcommand{\Sum}{\displaystyle \sum}

\newcommand{\Int}{\displaystyle \int}

\newcommand{\Sup}{\displaystyle \sup}

\numberwithin{equation}{section}

\newcommand{\beq}{\begin{eqnarray}}
\newcommand{\eeq}{\end{eqnarray}}
\newcommand{\bq}{\begin{equation}}
\newcommand{\eq}{\end{equation}}
\newcommand{\beqn}{\begin{eqnarray*}}
\newcommand{\eeqn}{\end{eqnarray*}}
\newcommand{\bex}{\begin{exo}}
\newcommand{\eex}{\end{exo}}
\newcommand{\ben}{\begin{enumerate}}
\newcommand{\een}{\end{enumerate}}

\let\b=\beta

\let\de=\delta

\def\cI{{\mathcal I}}
\def\cJ{{\mathcal J}}

\def\na{\nabla}

\def\div{\mathop{\mathrm{div}}\nolimits}

\def\bT{\text{\bf T}}
\def\bc{\text{\bf c}}
\def\bD{\text{\bf D}}
\def\b0{\text{\bf 0}}
\def\bI{\text{\bf I}}
\def\fff{\text{\bf G}}
\def\bQ{\text{\bf Q}}
\def\bH{\text{\bf H}}
\def\bnul{\text{\bf 0}}
\def\tr{\text{tr}}

\author{Miroslav Bul\'{\i}\v{c}ek}
\address{Mathematical Institute of Charles University, Sokolovsk\'{a} 83
186 75 Prague 8, Czech Republic.}
\email{mbul8060@karlin.mff.cuni.cz}
\thanks{Miroslav Bul\'{\i}\v{c}ek is supported by
Ne\v{c}as Center for Mathematical Modeling, project LC06052
financed by M\v{S}MT.}

\author{Mohamed Majdoub}
\address{University of Tunis ElManar, Faculty of Sciences of Tunis,
Department of Mathematics.}
\email{mohamed.majdoub@fst.rnu.tn}
\thanks{M. Majdoub is grateful to the Laboratory of
PDE and Applications at the Faculty of Sciences of Tunis.}

\author{Josef M\'{a}lek}
\address{Mathematical Institute of Charles University, Sokolovsk\'{a} 83
186 75 Prague 8, Czech Republic.}
\email{malek@karlin.mff.cuni.cz}
\thanks{The contribution of Josef M\'{a}lek to this work is a
part of the research project MSM 0021620839 financed by
M\v{S}MT. J. M\'{a}lek thanks the Czech Science Foundation, the
project GA\v{C}R 201/09/0917, for its support.}

\title[Unsteady Flows of Fluids with Pressure Dependent...]
{Unsteady Flows of Fluids with Pressure Dependent Viscosity in Unbounded Domains}

\date{\today}

\begin{document}


\begin{abstract}
In order to describe behavior of various liquid-like materials
at high pressures, incompressible fluid models with pressure
dependent viscosity seem to be a suitable choice. In the
context of implicit constitutive relations involving the Cauchy
stress and the velocity gradient these models are consistent
with standard procedures of continuum mechanics. Understanding
mathematical properties of governing equations is connected
with various types of idealization, some of them lead to
studies in unbounded domains. In this paper, we first bring up
several characteristic features concerning fluids with pressure
dependent viscosity. Then we study three-dimensional flows of a
class of fluids with the viscosity depending on the pressure
and the shear rate. By means of higher differentiability
methods we establish large data existence of a weak solution
for the Cauchy problem. This seems to be a first result that
analyzes flows of considered fluids in unbounded domains. Even
in the context of purely shear rate dependent fluids of a
power-law type the result presented here improves some of
earlier works.
\end{abstract}





\maketitle



\section{Introduction}


Navier-Stokes equations, a widely used model describing flows of incompressible Newtonian fluids, are usually expressed in terms of the velocity $\vv=(v^1,v^2,v^3)$ and the pressure $p$ in the form
\begin{equation}
\rho^* (\vv_{,t} + [\nabla \vv]\vv) = -\na p+\nu^{*}\,\Delta\vv,\quad \div\vv=0, \label{I1}
\end{equation}
where $\rho^*$ and $\nu^*$ are positive constants representing the density and the viscosity of a fluid, and
$\vv_{,t} + [\nabla \vv]\vv := \frac{\partial\vv}{\partial t}+\sum_{k=1}^3\,\frac{\partial\vv}{\partial x_k}\,v^k$.

Within the context of continuum mechanics, the system \eqref{I1} appears as a consequence of the constraint of incompressibility
\begin{equation}
\div\vv=0, \label{I2}
\end{equation}
and the balance of linear momentum
\begin{equation}
\rho^* (\vv_{,t} + [\nabla \vv]\vv) =\div\mbox{\bf T},
\label{I3}
\end{equation}
where the Cauchy stress {\bf T} is related to the symmetric part $\dd$ of the velocity gradient $\na\vv$ through the constitutive equation
\begin{equation}
\mbox{\bf T}=-p\, \mbox{\bf I}+2\nu^*\dd. \label{I4}
\end{equation}
Note that due to \eqref{I2} $p$ is the mean normal stress, i.e. $p=-\frac{1}{3}\,\mbox{tr}\,\mbox{\bf T}$. Setting $\bs:=2\nu^* \,\dd $ it is obvious that $\bs$ is the deviatoric part of $\mbox{\bf T}$.\\

Stokes, in his famous paper \cite{stokes} that can be considered as one of the corner-stone of continuum fluid mechanics, carefully discusses the possibility that the viscosity of a fluid may vary with the pressure, see also Hron et al. \cite{hmr00} for details. Barus \cite{barus1893}, and later on Andrade \cite{andrade30} (see also the book by Bridgman \cite{bridgman31}) experimentally showed that the viscosity grows with increasing pressure exponentially. Further details and references to more recent experimental studies can be found in the book by Szeri \cite{szeri98} and in the paper by M\'alek, Rajagopal \cite{handbook2} that reflects the situation before 2006. Recent papers\footnote{Further related experimental studies that concentrates on the dependence of $\nu$ on the pressure and the temperature are presented in Casalini and Bair \cite{casalini} and Harris and Bair \cite{harris}.} by Bair and Kottke \cite{bair03} and by Bair \cite{bair06} report even drastically faster dependence of the viscosity on the pressure. It is worth of noticing that even at such higher pressures the variation in the density of most liquids, in comparison to the variations in the viscosity, is negligible, as discussed in Rajagopal \cite{RajImpl05} or \cite{handbook2}. As a consequence, one can model these liquids as incompressible materials fulfilling the constitutive equation
\begin{equation}
\mbox{\bf T}=-p\, \mbox{\bf I}+2\nu (p) \dd. \label{I8}
\end{equation}
\\

There are many experimental works going back to observations made by Trouton in 1906 \cite{trouton1906} and Schwedoff around 1890 (as discussed in \cite{TannerWalters1998}) that confirm the dependence of the viscosity on the shear rate. A representative list of references can be found in the most of books on non-Newtonian fluids, see Schowalter \cite{schowalter}, Bird, Amstrong, Hassager \cite{bird77}, Huilgol \cite{huilgol} or the survey papers by M\'alek, Rajagopal, R\accent23u{\v{z}}i{\v{c}}ka \cite{MRR95} and M\'alek, Rajagopal \cite{mr05}.\\

There is less available experimental data that characterize the dependence of the viscosity both on the shear rate and on the pressure. This is very likely due to completely different experimental set-up used for the measurement of the relation between the shear rate and the shear stress on the one hand, and for the measurement of the viscosity-pressure relationships on the other hand. Nevertheless, we can provide references where the incompressible fluids with the pressure and shear rate dependent viscosity are chosen in order to model high pressure processes in silos and journal bearings, and where some support of experimental observations to the considered form of the viscosities is available.

Schaeffer \cite{schaeffer87} used  the model \eqref{I2}--\eqref{I4} with
\begin{equation}
\nu\left(p,|\dd|^2\right)=\alpha\,p\,|\dd|^{-1}\qquad (\alpha>0) \label{I5}
\end{equation}
in modeling and analyzing flows of certain granular materials. Davies and Li  \cite{DaviesLi94} and Gwynllyw, Davies and Phillips \cite{GDP96} (see also \cite{LGDP00}) considered a model for computational simulations of pressure and temperature effects in certain viscoelastic liquids, in which the viscosity is of the form
\begin{equation}
\nu\left(p,|\dd|^2\right)
=\Big(\eta_\I+\frac{\eta_0-\eta_\I}{1+\beta\,|\dd|^{2-r}}\Big)\,\exp
(\alpha p) \label{I6}
\end{equation}
with $r=1.46$  and  $\alpha, \beta, \eta_0, \eta_\I
>0$.

To summarize there are experimental data confirming a (linear,
exponential, or even faster) dependence of the viscosity on the
pressure, and also on the shear rate. Thus, neglecting less
significant variations in the density of a liquid at high
pressure one end-up with the model \eqref{I2}, \eqref{I3} and
the constitutive equation
\begin{equation}
\mbox{\bf T}=-p\, \mbox{\bf I} + \nu\left(p,|\dd|^2\right)\,\dd.\label{I7}
\end{equation}
In what follows, we shall show that \eqref{I7} is consistent with standard procedures of continuum mechanics if one starts with a general implicit relation between $\mbox{\bf T}$ and $\dd$.

\subsection{Fluids with pressure dependent viscosities within implicit constitutive theory}


A standard 'derivation' of the constitutive equation for a compressible Navier-Stokes
fluid starts with the assumption that the Cauchy stress $\bT$ in a fluid depends
on the density and the velocity
gradient $\nabla \vv$, i.e.,
\begin{equation}
   \bT = \bH(\varrho, \nabla \vv)\,. \label{Compr1}
\end{equation}
The requirement that $\bT$ is invariant with respect to any coordinate transformation leads to the conclusion (see for example Truesdell \cite{truesdell91} or Serrin \cite{serrin63}) that
$\bH$ depends on the velocity gradient through its symmetric part $\bD(\vv)$ and $\bH$ is an isotropic second order tensor. This means that
$$
\bH (\varrho, Q \bD \bQ^{T})=\bQ \bH(\varrho, \bD) \bQ^{T} \quad \forall \bQ \in  \mathcal{Q}\,,
$$
where $\mathcal{Q}$ denotes the orthogonal group. Using the representation theorem for such isotropic second order tensors one obtains that
$$
\bT= \alpha_1 (\varrho, I_{\bD}, II_{\bD},  III_{\bD})\bI +
\alpha_2(\varrho, I_{\bD}, II_{\bD},  III_{\bD})\bD + \alpha_3(\varrho, I_{\bD}, II_{\bD},  III_{\bD}) \bD^2,
$$
where
$$
   I_{\bD}= \tr \bD, II_{\bD}=  \tr \bD^2, III_{\bD}= \tr \bD^3.
$$
If one requires that the  stress depend linearly on $\bD$, then
one obtains the constitutive equation for the classical compressible Navier-Stokes fluid, namely
\begin{equation}
  \bT = -p(\varrho) \bI + 2 \mu(\varrho) \bD + \lambda (\varrho) (\tr \bD) \bI\,. \label{Compr2}
\end{equation}
This straightforward procedure (that derives the constitutive equation for an compressible Navier-Stokes fluid
starting from the assumption \eqref{Compr1} and followed by some standard requirements) however loses its clarity
if one aims to arrive at an incompressible Navier-Stokes fluid. We recall that the classical approach that is employed in most continuum mechanics textbooks to enforce internal constraints (such as the constraint of incompressibility \eqref{I2}) is to require that the constraints do no work. Then the possibility that the viscosity depends on the pressure is excluded, and one cannot deduce the constitutive equation like \eqref{I8} or \eqref{I7}. We refer to Rajagopal and Srinivasa \cite{RaSi05} for more details and for outlining how to overcome this drawback (within the framework of explicit constitutive equations of the form \eqref{Compr1} in purely mechanical context) by relaxing the requirement that the constraints do no work).

It is thus interesting to observe that the procedure starting with the assumption \eqref{Compr1} and ending with the compressible Navier-Stokes fluids \eqref{Compr2} is also applicable to incompressible Navier-Stokes fluids provided that we start with a general implicit constitutive equation\footnote{Implicit constitutive theory, as introduced in \cite{RajImpl03} and \cite{RajImpl05} (see also \cite{RS08Impl} for further development and \cite{Malek_ETNA} for a reflection of the implicit constitutive theory on mathematical analysis of non-linear PDE's), provides the framework that is sufficiently robust to capture complicated response of materials. In this approach the quantities such as stress and strain in solid-like models and stress and the velocity gradient in fluid-like models share an equivalent role (on contrary to the explicit constitutive theory that usually prefers strain to stress since stress is a function of the strain or the velocity gradient). In addition, the implicit constitutive theory can eliminate some internal variable theories that are frequently connected with less clear physical meaning and with difficulties to identify appropriate boundary conditions.} between $\mbox{\bf T}$ and $\dd$ of the form
\begin{equation}
    \fff(\bT, \bD) = \bnul \,. \label{Incompr1}
\end{equation}
Analogously as above, if we require the
function $\fff$ to be isotropic, then $\fff$ has to satisfy the
restriction
$$
\fff(\varrho, \bQ \bT \bQ^{T}, \bQ \bD \bQ^{T})=\bQ \fff(\varrho, \bT, \bD)\bQ^{T} \quad \forall \bQ \in  \mathcal{Q}\,.
$$
The representation theorem relevant to such isotropic tensors  immediately implies (see Spencer \cite{spencer1971}) that
\begin{equation}
\begin{split}
\alpha_0 \bI + \alpha_1 \bT +\alpha_2 \bD + \alpha_3 \bT^2 + \alpha_4 \bD^2 + \alpha_5(\bD \bT + \bT \bD) + \alpha_6(\bT^2\bD + \bD \bT^2)\\
+\alpha_7 (\bT \bD^2 + \bD^2 \bT) + \alpha_8 (\bT^2 \bD^2 + \bD^2 \bT^2) = \bnul\,,
\end{split} \label{Incompr2}
\end{equation}
where $\alpha_i, i=0, \ldots 8$ depend on the invariants
$$
\tr \bT, \tr \bD, \tr \bT^2, \tr \bD^2, \tr \bT^3, \tr \bD^3, \tr(\bT \bD), \tr (\bT^2 \bD), \tr(\bD^2 \bT), \tr(\bT^2 \bD^2).
$$
We note that if
\begin{align*}
\alpha_0&= - \frac13 \tr \bT,\\
\alpha_1&= 1,
\intertext{and} \alpha_2&= -2 \nu (-\frac13 \tr \bT, |\bD|^2)\qquad (\mu>0),
\end{align*}
we obtain \eqref{I7}.

Also, requiring that \eqref{Incompr2} is linear both in $\bD$ and both in $\bT$, one obtains
$$
  \gamma_0 \, (\tr \bT) \, \bI + \gamma_1 \, (\bD\cdot \bT) \, \bI + \gamma_2 \bT + \gamma_3 \, (\tr \bT) \, \bD + \gamma_4 (\bT \bD + \bD \bT) = \bnul\,,
$$
where $\gamma_i$, $i=0,1,2,3,4$, are constants. Requiring that the constraint of incompressibility holds, i.e. $\tr \bD = 0$, the last identity implies that ($a_i$, $i=1,2,3$, are constants)
$$
  a_1 (\bT - \frac13 \tr \bT \, \bI) + a_2 (\frac13 \tr \bT) \bD + a_3 (\bT\bD + \bD\bT - 2 \bT \cdot \bD) = \bnul\,.
$$
Setting $a_3=0$ we obtain a model where the viscosity depends on $p$ linearly (similarly as in the Schaeffer model \eqref{I5}).

\subsection{On flows in unbounded containers and relevant results established earlier}

Any real flow takes place in a bounded domain. Nevertheless, in order to understand the mathematical properties of the equations that govern fluid motions and to observe how solutions relevant to a material $A$
can differ from solutions of the system of equations describing flow of another material $B$ one frequently comes
across unbounded domains. Flows between two infinite parallel plates (plane Couette-Poiseuille flows), flows in a tube, flows glowing down an inclined plane, flows due to suddenly accelerated plate or flows due to an oscillating plane are examples of motions that are considered in unbounded domains. Despite their unboundedness the shapes of these domains are very similar to containers where properties of flows are experimentally measured. Then the analytical solutions can be compared with experimental data and in most cases the coincidence is very good. Several studies that
aim to obtain some explicit solutions are in place.

Hron et. al. \cite{hmr00} investigate the possibility to find explicit solutions for flows between two infinite parallel plates with no-slip boundary conditions and for the viscosities of two types: (i) $\nu(p) = \exp(\alpha p)$ and (ii) $\nu(p,|\bD|^2) = \alpha \, p\, |\bD|^{r-2}$ for $r\in (1,2\rangle$. (See also relevant study \cite{Susl1} where however some imprecise statements are made, and \cite{Susl2} focused on some stability issues.)
Vasudevaiah and Rajagopal \cite{krraplmat05} considered the fully
developed flow in a pipe dealing with a fluid that has a viscosity that depends on the
pressure and shear rate and were able to obtain explicit exact
solutions for the problem. Kannan and Rajagopal
\cite{KanKRR04} analyzed effects of gravity on flows between rotating parallel plates. For various viscosity-pressure relationships they observed, among others, that a boundary layer can be adjacent to just one of the plate. Massoudi and Phuoc \cite{Massoudi06} considered flows between parallel plates due to oscillatory pressure gradient.
Rajagopal \cite{KRR08} finds special solutions for flows down an inclined plane.
Srinivasan and Rajagopal in \cite{ShriramKRR_IJES} study flow of fluids with pressure dependent viscosity
due to a suddenly accelerated plate and due to an oscillating plate.
In all these flows the velocity field, and consequently also the structure of the vorticity
and the shear stresses at the walls are significantly different from those for the classical Navier-Stokes fluid.

In mathematical analysis, in order to separate the problems connected with
governing equations from those that are due to the presence of the boundary, it is preferable to look for solution either in the whole space\footnote{Leray's fundamental paper \cite{Le34} on long-time and large-data existence theory for the Navier-Stokes equations concerns the Cauchy problem.} or in spatially periodic setting; the former corresponds to the Cauchy (or initial-value) problem. The Cauchy problem for flows of a one class of fluids with pressure and shear rate dependent viscosity is the topic investigated in this paper. Regarding the approach we carry on the
regularity method employed by M\'{a}lek, Ne\v{c}as and Rajagopal
\cite{MNR} for $\bs(p,\dd(\vv))$ and spatially
periodic problem, and on Pokorn\'{y} work \cite{Pok} dealing with
the Cauchy problem for $\bs(\dd(\vv))$, being
independent of $p$. However, we have to strengthen several steps in Pokorn\'y approach in order to make the whole procedure applicable to fluid with pressure and shear rate dependent viscosity. As a consequence, some results presented here seem to be new even for fluids with shear rate dependent viscosity.

Before going to details we survey results that investigate mathematical properties of flows of such fluids
in general, but compact domains.


To date there have been few mathematically rigorous studies
concerning fluids with pressure dependent viscosity. To our
knowledge, there is no global existence theory that is in
place for both steady and unsteady flows of fluids whose viscosity
depends purely on the pressure. Previous studies by Renardy
\cite{renardy86}, Gazzola \cite{gazzola97} and Gazzola and Secchi
\cite{gazzola98} either addressed existence of solutions that are
short-in-time and for small data or assumed structures for the
viscosity that are contradicted by experiments.  Recently, there has
been some resurgence of interest in studying the flows of fluids with pressure
dependent viscosities. M\'{a}lek, Ne\v{c}as and Rajagopal \cite{MaNeRa02}, Hron,
M\'{a}lek, Ne\v{c}as and Rajagopal \cite{hmnr03} and Franta, M\'{a}lek
and Rajagopal \cite{malek04} have established existence results
concerning the flows of fluids whose viscosity depends on both the
pressure and symmetric part of the velocity gradient in an
suitable manner. Franta et al. \cite{malek04} established
the existence of weak solutions for the {\it steady} flows of fluids whose
viscosity depends on both pressure and the symmetric part of the
velocity gradient, that satisfy Dirichlet boundary conditions. This result has been extended to lower values of power-law exponent in \cite{BF09}.
Earlier, M\'{a}lek  et al. \cite{MNR} and Hron et al. \cite{hmnr03}
established global-in-time existence for unsteady flows of such
fluids under spatially periodic boundary conditions. The extension
of these results to flows in bounded domains subject to the Navier's slip are
due to Bul\'\i{}\v{c}ek, M\'{a}lek and Rajagopal \cite{BuMaRa07} and to fully thermodynamical setting
in Bul\'\i{}\v{c}ek, M\'{a}lek and Rajagopal \cite{bmr08A}. There are models where the viscosity goes to $\infty$ as
$p\to \infty$ and the existence of weak solutions can be established, \cite{bmr08B}. The relation of models and results to implicit constitutive theory is discussed in a survey paper \cite{Malek_ETNA}.

\section{Setting of the Problem and the Main Result}


We are interested in studying the following Cauchy problem: to find
$(\vv,p) : [0,T]\times\R^3\to \R^3\times\R$ solving (in a weak sense)
\begin{equation}
\label{eq1} \vv_t+\div(\vv\otimes\vv)-\div(\bs(p,\dd(\vv)))=-\nabla p,\quad\mbox{in}\quad [0,T]\times\R^3
\end{equation}
\begin{equation}
\label{eq2}
\div\vv=0,\quad\mbox{in}\quad[0,T]\times\R^3
\end{equation}
\begin{equation}
\label{CD} \vv(0,\cdot)=\vv_0\quad\mbox{with}\quad\div\vv_0=0,
\end{equation}
where $T$ is the length of the time interval of interest, $v_0$ is
given divergenceless vector field (initial velocity) and $\bs(p,\dd(\vv))$ stands for the constitutively determined part
of the Cauchy stress $\T$ being of the form $\T=-p\II+\bs(p,\dd(\vv))$. Typically, we
consider the case $\bs(p,\dd(\vv))=\nu(p,|\dd(\vv)|^2)\dd(\vv)$, allowing thus the possibility that the
viscosity depends on the pressure $p$ and the quantity $|\dd(\vv)|^2$, that at simple shear flows simplifies to the shear
rate.\\

As regards the structure, we assume that for any $1\leq r\leq 2$
there exists $C_1,\;C_2>0$ such that for all $p\in\R$ and all
symmetric matrices $\bb,\;\dd\in\R^{d\times d}$
\begin{align}
\tag{$\mathbf{A1}$} &C_1(1+|\dd|^2)^{\frac{r-2}{2}}\,|\bb|^2\leq \frac{\partial\bs(p,\dd)} {\partial \dd}\,\cdot\,\left(\bb\otimes\bb\right)\leq C_2(1+|\dd|^2)^{\frac{r-2}{2}}\,|\bb|^2 \label{A1}\\
\intertext{
and there is a $\gamma_0>0$ such that for all $p\in\R$ and all
$\dd\in\R^{d\times d}$}
\tag{$\mathbf{A2}$}&
\Big| \frac{\partial\bs(p,\dd)}{\partial p}\Big|\leq \gamma_0 (1+|\dd|^2)^{\frac{r-2}{4}}\leq\gamma_0.\label{A2}
\end{align}
As a consequence of \eqref{A1}-\eqref{A2} we get
\begin{lem}
\label{ineq-visc}
Let assumption $(\as)$ hold. Then for all $p\in\R$ and $\dd\in \R^{3\times 3}_{sym}$
\begin{equation}
\label{visc1}
\bs(p,\dd) \cdot \dd \geq C_1 |\dd|^2(1+|\dd|^2)^{\frac{r-2}{2}}\ge\frac{C_1}{2} \min (|\dd|^2,|\dd|^r)
\end{equation}
and
\begin{equation}
\label{visc2}
\Big|\bs(p,\dd) \Big|\leq C_2|\dd|(1+|\dd|^2)^{\frac{r-2}{2}}\le C_2
\min(|\dd|,|\dd|^{r-1}).
\end{equation}
\end{lem}
\begin{proof} See Lemma 1.19 p. 198 in \cite{MNRR}.
\end{proof}
The aim of this paper is to establish the following existence result.
\begin{thm}
\label{MAIN} Let $\vv_0\in L^{2}_{div}$ and $\bs$ fulfil
the assumptions \eqref{A1}--\eqref{A2}  with $\gamma_0<\frac{C_1}{C_1+C_2}$. Assume that
$r\in(\frac{9}{5},2)$ then there is a couple $(\vv,p)$ such that
\begin{align}
\vv&\in L^{\infty}(0,T;L^2(\R^3)^3) \cap L^r(0,T; D^{1,\phi}(\R^3)^3)\cap W^{1,\frac{5r}{6}}(0,T; W^{-2,2}(\R^3)^3),\label{THEq1}\\
\intertext{with $\phi(s):=s^2(1+s^2)^{\frac{r-2}{2}}$ and}
p_1& \in L^q(0,T;L^q(\R^3)), \quad \textrm{ for all } q\in (1,\frac{5r}{6}),\label{THEq2}\\
p_2 &\in L^s(0,T;L^s(\R^3)), \quad \textrm{ for all } s\in [2,r'], \label{THEq3}
\end{align}
where $p$ is given as $p=p_1+p_2$, and that solves
for all $\bpsi\in L^{\frac{5r}{5r-6}}(0,T;W^{2,2}(\R^3))$
\begin{equation}
\Int_0^T\langle \vv_t,\bpsi\rangle-(\vv \otimes \vv, \nabla \bpsi)+(\bs(p,\dd(\vv)),\dd(\bpsi)) \; dt
=\Int_0^T\,(p,\div\bpsi)\; dt.
\label{weak21}
\end{equation}
\end{thm}
Note that all spaces used in the previous theorem are introduced in the next section.\\

Nowadays, there are several results concerning the model in
investigation. All of them are treating flows on bounded domains.
The novelty in our study is that we consider unbounded domain
(especially $\R^3$). Regarding the approach we carry on the
regularity method employed by M\'{a}lek, Ne\v{c}as and Rajagopal
\cite{MNR} for $\bs(p,\dd(\vv))$ and spatially
periodic problem, and on Pokorn\'{y} work \cite{Pok} dealing with
the Cauchy problem for $\bs(\dd(\vv))$, being
independent of $p$.


\section{Basic definitions and auxiliary lemmas}


In this section we will fix the notations, state the basic
definitions and recall some known and useful tools. In order to distinguish between scalar-, vector- and tensor-valued function we denote by bold symbol  the vector-valued function, i.e. $\vv:=(v^1,v^2,v^3)$, and capital bold symbol for tensor valued function, i.e., $(\bs)_{ij}:= S_{ij}$ for $i,j=1,2,3$. Let $\phi: \mathbb{R}_+ \to \mathbb{R}_+$ be an increasing continuous convex function that  vanishes at zero. For arbitrary open $\Omega \subset \mathbb{R}^3$ we denote by the symbol $L^{\phi}(\Omega)$ the Banach space
$$
L^{\phi}(\Omega):=\overline{\mathcal{D}(\Omega)}^{\|\cdot \|_{L^\phi}} \quad \textrm{ with } \quad \|v\|_{L^\phi}=\inf \left\{\lambda >0; \int_{\Omega} \phi\left(\frac{|v|}{\lambda}\right) \; dx \le 1 \right\}.
$$
Note that if $\phi(s)=s^r$ with some $r\in [1,\infty)$, we write $L^r(\Omega):=L^{\phi}(\Omega)$ and this definition corresponds to the standard one for Lebesgue spaces. The space $L^{\infty}(\Omega)$ is defined by usual way. The Sobolev spaces $W^{k,p}(\Omega)$ are defined through usual way. In addition, for arbitrary $k\in \mathbb{N}_0$ and arbitrary $\phi$ satisfying the condition mentioned above we define the space $D^{k,\phi}(\Omega)$ as
$$
D^{k,\phi}(\Omega):= \overline{\mathcal{D}(\Omega)}^{\|\cdot \|_{D^{k,\phi}}} \quad \textrm{ with } \quad \|v\|_{D^{k,\phi}}:=\|\nabla^k v\|_{L^\phi}.
$$
Similarly as above, the symbol $D^{k,r}(\Omega)$ with $r\in [1,\infty)$ denotes the space $D^{k,\phi}(\Omega)$ with $\phi(s):=s^r$. Since we work with solenoidal functions, we also denote
$$
W^{k,r}_{div}(\Omega):=\{\vv \in W^{k,r}(\Omega)^3; \div v =0\},
$$
where we used the notation $X^m:=\underset{m-\textrm{times}}{\underbrace{X\times \cdots \times X}}$. As usually, for a Banach space $X$, the symbols $L^r(0,T;
X)$ and ${\mathcal C}(0,T;X)$ stand for the standard Bochner
spaces. Finally, if $a\in L^r(\mathbb{R}^3)$ and $b\in L^{r'}(\mathbb{R}^3)$, we define $(a,b):=\int_{\mathbb{R}^3} \,a(x)\,b(x)\;dx$, and in addition for general Banach space $X$ we also set $\langle a,b\rangle:=\langle a,b\rangle_{X, X^*}$ whenever $a\in X$ and $b\in X^*$ and whenever the meaning of the duality is clear. We should note that in the rest of the paper the functions $\phi,\;\psi:\mathbb{R}_+ \to \mathbb{R}_+$ are defined as $\phi(s):=s^2(1+s^2)^{\frac{r-2}{2}}$ and $\psi(s):=\min(s^6,s^\frac{3r}{3-r})$.

In addition we recall all auxiliary lemmas needed in the following text. First, since we will need to deal with Orlicz spaces, we will frequently use the embedding theorem.
\begin{lem}\label{LM3.1}
 There there exists a constant $C>0$ such that
\begin{equation}
\|v\|_{L^\psi}\le C \|\nabla v\|_{L^\phi}
\label{emb}
\end{equation}
for all $v\in D^{1,\phi}(\mathbb{R}^3)$.
\end{lem}
\begin{proof}
The proof for more general $\phi$ can be found in \cite[Theorem 1, p. 432]{Ci04}.
\end{proof}
The next lemma is in fact the Korn inequality in some Orlicz spaces.
\begin{lem}[Korn inequality]\label{LKorn}
Let $r\in (\frac{6}{5},2)$. There exists a constant $C>0$ such that for all $\vv \in D^{1,\phi}(\mathbb{R}^3)^3$ there holds
\begin{equation}
\|\nabla v\|_{L^\phi}\le C \|\ \dd (\vv) \|_{L^\phi}.
\label{Korn}
\end{equation}
\end{lem}
\begin{proof}
Since $\phi$ is $\mathcal{C}^1$ convex functions which satisfies $\triangle_2$ condition and its complementary\footnote{The definition of complementary function and also of $\triangle_2$ condition can be found eg. in \cite{DiRuSc08}. Note that our $\phi$ satisfies all the assumptions of the Theorem 6.4 in \cite{DiRuSc08} because $1<r<\infty$.} function as well, the inequality \eqref{Korn} is a consequence of \cite[Theorem 6.4]{DiRuSc08}.
\end{proof}
As a simple consequence of \eqref{emb} and \eqref{Korn} one can deduce the following
\begin{cor}
 Assume that $q\in [2,\frac{5r}{3}]$. There is a constant $C$ such that for arbitrary $\vv$ satisfying
$$
\|\vv\|_{L^{\infty}(0,T;L^2(\R^3)^3)} + \|\dd (\vv)\|_{L^r(0,T;L^{\phi}(\R^3)^{3\times 3})} \le K
$$
there holds
\begin{equation}
\|\vv\|_{L^q(0,T;L^q(\R^3)^3)}\le CK. \label{INT}
\end{equation}
\end{cor}
\begin{proof}
First, using \eqref{emb} and \eqref{Korn}, we conclude that $\|\vv\|_{L^r(0,T;L^{\psi}(\R^3)^3)} \le CK$. Hence, we can estimate (with $\vv_\psi=\|v\|_{L^\psi}$)
\begin{align*}
\int_0^T \int_{\mathbb{R}^3} |\vv|^q \; dx \; dt &=\int_0^T \int_{|\vv|\le \vv_\psi} |\vv|^q \; dx + \int_{|\vv| >\vv_\psi} |\vv|^q \; dx \; dt\\
&\le \int_0^T \|\vv\|_2^{\frac{6-q}{2}}\left(\int_{|\vv|\le\vv_\psi}|\vv|^6 \; dx \right )^{\frac{q-2}{4}}\\
&\qquad+ \|\vv\|_2^{\frac{6r-69+2rq}{q(5r-6)}} \left( \int_{|\vv|>\vv_\psi} |\vv|^{\frac{3r}{3-r}} \; dx \right)^{\frac{3-r}{3r}\frac{3r(q-2)}{5r-6}}\; dt\\
&\le C\int_0^T  \|\vv\|_2^{\frac{6-q}{2}} \|\vv\|_{L^\psi}^{\frac{3(q-2)}{2}}+ \|\vv\|_2^{\frac{6r-69+2rq}{q(5r-6)}} \|\vv\|_{L^\psi}^{\frac{3r(q-2)}{5r-6}}\; dt \le CK
\end{align*}
provided that $q\le \frac{5r}{3}$.
\end{proof}

Next, we use the convention that the solution to  the following  Laplace equation
\begin{equation}
\triangle u = f \qquad \textrm{ in } \mathbb{R}^3, \label{laplace}
\end{equation}
is always given by
\begin{equation}
u(x):=\frac{1}{4\pi} \int_{\mathbb{R}^3} \frac{f(y)}{|x-y|} \; dy,
\label{deflap}
\end{equation}
whenever the integral on the right-hand side is meaningful. Using Calder\'{o}n-Zygmund singular integral operator theory, we can conclude that for all $q\in (1,\infty)$
\begin{align}
\|\nabla^2 u\|_q &\le K_q \|f\|_q,\label{2g}\\
\|\nabla u \|_q &\le K_q \|\vv\|_q, &&\textrm{with } f=\div \vv, \label{1g}\\
\|u\|_q &\le K_q \|\bs\|_q, &&\textrm{with }f=\div \div \bs. \label{0g}
\end{align}
Moreover, it is easy to show that we can set $K_2=1$. In addition, the solution of \eqref{laplace} given by \eqref{deflap} is unique in the class of function that solves \eqref{laplace} and that vanish at infinity.
Having theory for Laplace equation, we introduce the  so-called Helmholtz decomposition
$$
\vv=\vv_{div}+\nabla g^{\vv},
$$
where for a given $\vv\in W^{1,r}(\mathbb{R}^3)^3$, $g^{\vv}$ solves the problem
\begin{equation}
\label{helm} \Delta g^{\vv}=\div\,\vv,\quad g^{\vv}(x)\to
0\quad\mbox{as}\quad |x|\to\infty, \end{equation} and
$\vv_{div}:=\vv-\nabla g^{\vv}$. The following estimates then easily follows from \eqref{2g}--\eqref{1g}
\begin{equation}
\label{helm1} \|\nabla^2 g^{\vv}\|_{q} \le K_q \|\div\,\vv\|_q,\quad
\|\nabla \vv_{div}\|_{q}\le K_q \|\nabla \vv\|_{q},
\end{equation}
\begin{equation}
\label{helm2} \|\nabla g^{\vv}\|_{s} \le K_s \|\vv\|_s,\quad
\|\vv_{div}\|_{s}\le K_s \|\vv\|_{s},
\end{equation}
whenever the right hand sides make sense.

Next, we define an orthonormal  basis which is essential to perform a Fadeo-Galerkin scheme. Recall that for {\bf bounded} domain such a basis can be formed by the eigenvectors of the operator $-\Delta$ with suitable boundary conditions. Unfortunately, for $\R^3$ this procedure fails since the Laplacian has no eigenvalue. Using the fact that ${\mathcal
D}(\R^3)$ is separable we have:
\begin{lem}
\label{basis}
There is an orthonormal basis $\mathbf{\mathcal{B}}=\{\bo^1,\bo^2,\cdots\}$ of $W^{2,2}_{\div}(\R^3)$ such that each $\bo^k$ belongs to ${\mathcal
D}(\R^3)$.
\end{lem}
\begin{proof}
For the convenience of the reader, we outline the proof. Since ${\mathcal
D}_{\div}(\R^3)$ is dense in $W^{2,2}_{\div}$, and
${\mathcal D}_{\div}(\R^3)$ is separable, there is a countable subset $\mathbf{\mathcal{C}}=\{\phi^1,\phi^2, \cdots\}$ of ${\mathcal D}(\R^3)$ dense in $W^{2,2}_{\div}$. From $\mathbf{\mathcal{C}}$ we can construct an orthonormal basis $\mathbf{\mathcal{B}}=
\{\bo^1,\bo^2,\cdots\}$ of the Hilbert-space $W^{2,2}_{\div}$ following the Gram-Schmidt procedure.
\end{proof}

Finally,  we review the classical Caratheodory theory for ordinary differential equations
\begin{equation}
\label{ODE-C}
\dot{\xi}(t) = {\mathcal A}(t, \xi(t))\quad \mbox{\sf a.e.},\quad \xi(0) = x,
\end{equation}
with non-smooth right-hand side ${\mathcal A} : I\times\R^3\to\R^3$
or equivalently the corresponding integral equation
\begin{equation}
\xi(t) = x +\Int_0^t\,{\mathcal A}(s, \xi(s))\,ds.
\end{equation}

\begin{defi}
\label{cara}
Let ${\mathcal A} : I\times\R^3\to\R^3$ be a function, where $I\subset\R$ is an interval. We say that ${\mathcal A}$ satisfies the Caratheodory conditions ({\bf CC}) if the following holds:\\

\begin{itemize}
\item ${\mathcal A}(t,x)$ is Lebesgue measurable
in $t$ for all fixed $x\in\R^3$,\\

\item ${\mathcal A}(t,x)$ is continuous in $x$ for almost all $t\in I$, and\\

\item $\Sup_{x\in\R^3}\,|{\mathcal A}(t,x)|\leq \beta(t)$ {\sf a.e.} for some positive function $\beta\in L^1_{loc}(I)$.
\end{itemize}
\end{defi}
Note that the first two Caratheodory conditions ensure Lebesgue
measurability of the composition $s\to {\mathcal A}(s, f(s))$ for
all $f \in (C(I))^3$, while the third condition is crucial in the
existence proof.
\begin{thm}[{\sf Existence theorem for ordinary differential equations}]
Let $I$ be some subinterval of $\R$ and assume that ${\mathcal A}$
satisfies ({\bf CC}) on $I\times\R^3$. Let $x\in\R^3$, then there
exists an absolutely continuous solution $\xi= (\xi_1,\cdots,\xi_d)$
to the ordinary differential equation \eqref{ODE-C}.
\end{thm}


\section{Equation for the pressure} \label{eq-pressure}


This section is devoted to the solvability and the properties of the solution to the following equation
\begin{equation}
-\triangle p = \div \div \left ( \vv \otimes \vv - \bs (p,\dd(\vv))\right ). \label{pressureeq}
\end{equation}
The basic properties are then established in the following
\begin{prop}\label{Ppres}
Let $\bs$ satisfy \eqref{A1}--\eqref{A2} with $\gamma_0 <1$ and with $r\in (1,2)$. Assume that there are $2<q_1<\infty$ such that $\vv \in L^{q}(\R^3)^3$ for all $q\in [2,q_1]$, and assume that $\vv \in D^{1,\phi}(\R^3)^3$. Then there exists unique $p$ solving \eqref{pressureeq} such that $p=p_1+p_2$ where
\begin{align}
\|p_1\|_{\frac{q}{2}} &\le C(q) \|\vv\|^2_{q} &&\textrm{ for all } q\in (2,q_1], \label{p1est}\\
\|p_2\|_{s} &\le C(s)\|\dd(\vv)\|_{L^{\phi}} && \textrm{ for all } s\in [2,r']. \label{p2est}
\end{align}
In addition, if $\vv \in W^{2,2}(\R^3)^3$ then the following estimate holds
\begin{equation}
\|\nabla p\|_2 \le \frac{1}{1-\gamma_0} \left( C\|\na \vv\|_2 \|\na \vv \|_3 + C_2 \left(\int_{\R^3}(1+|\dd(\vv)|^2)^{r-2}|\dd(\na \vv)|^2\; dx\right)^{\frac12} \right). \label{gradp}
\end{equation}
Moreover, if there is a sequence $\{\vv^n\}_{n=1}^{\infty}$ satisfying
\begin{align}
\vv^n &\rightharpoonup \vv &&\textrm{ weakly in } L^q(0,T; L^q(\R^3)^3) \textrm{ for all } q\in [2,q_1],\label{opice1}\\
\vv^n &\rightarrow \vv &&\textrm{ a.e. in } (0,T)\times \R^3, \label{opice2}\\
\na \vv^n &\rightarrow \na \vv &&\textrm{ a.e. in } (0,T)\times \R^3,\label{opice3}\\
\|\vv^n\|_{D^{1,\phi}}&\le C &&\textrm{ uniformly w.r.t. } n, \label{opice4}
\end{align}
then there is a (not relabeled) subsequence $\{p^n\}_{n=1}^{\infty}$ solving\footnote{We use solvability of \eqref{pressureeq} at each time $t$.} \eqref{pressureeq} such that ($p^n=:p_1^n+p_2^n$)
\begin{align}
p_1^n &\rightharpoonup p_1 &&\textrm{ weakly in } L^{\frac{q}{2}}(0,T;L^{\frac{q}{2}}(\R^3)) \textrm{ for all } q\in (2,q_1],\label{cp1}\\
p_2^n &\rightharpoonup p_2 &&\textrm{ weakly in } L^{s}(0,T;L^{s}(\R^3)) \textrm{ for all } s\in [2,r'],\label{cp2}\\
p^n &\rightarrow p &&\textrm{ a.e. in } (0,T)\times \R^3,\label{cp}
\end{align}
and the couple $(p,\vv)$ solves \eqref{pressureeq}.
\end{prop}
Although the proof follows almost step by step the procedure developed in \cite{MNR} (see also \cite{BuMaRa07}), we prove it here in all details because in \cite{MNR,BuMaRa07} the similar results were proven only in bounded domain.
\begin{proof}[Proof of Proposition \ref{Ppres}]
We prove the existence of $p$ solving \eqref{pressureeq} only for $\vv \in \mathcal{D}(\R^3)^3$. The existence result for general $\vv$ then follows from \eqref{p1est}--\eqref{cp}. Hence, let $\vv\in \mathcal{D}(\R^3)^3$ be arbitrary. Define $p^1=0$ and we find $p^n$ as the solution of
\begin{equation}
-\triangle p^n = \div \div \left ( \vv \otimes \vv - \bs (p^{n-1},\dd(\vv))\right ). \label{pressureeqpr}
\end{equation}
Next, we show that $p^n$ is Cauchy sequence in $L^2(\R^3)$ and therefore converges strongly in $L^2(\R^3)$ to some $p$ that has to solve \eqref{pressureeqpr}. To do it, we first note that due to \eqref{A1} and \eqref{0g} $p^n\in L^2(\R^3)$ for all $n$. Next, since
$$
-\triangle (p^n-p^m) = \div \div \left (\bs(p^{n-1},\dd(\vv))-\bs(p^{m-1},\dd(\vv)) \right )
$$
the estimate \eqref{0g} implies that
$$
\|p^n-p^m\|_2\le  \|\bs(p^{n-1},\dd(\vv))-\bs(p^{m-1},\dd(\vv))\|_2 \overset{\eqref{A2}}\le \gamma_0 \|p^{n-1} - p^{m-1}\|_2.
$$
Since $\gamma_0 <1$ we end by using the Banach fixed point theorem.

Next, for a given $(p,\vv)$ solving \eqref{pressureeq} we define $p_1$, $p_2$ through
\begin{align}
-\triangle p_1 &= \div \div (\vv \otimes \vv),\label{opicka1}\\
\triangle p_2 &= \div \div \bs(p,\dd(\vv))\label{opicka2}.
\end{align}
Therefore we see that $p_1+p_2=p$ and \eqref{0g} implies (after using \eqref{visc2}, the assumption on $\vv$ and the fact that $1<r<2$) \eqref{p1est}--\eqref{p2est}. Next, using \eqref{1g}, we immediately deduce that
\begin{align*}
\|\nabla p\|_2 &\le \| |\vv| |\nabla \vv|\|_2 + \|\na \bs(p,\dd(\vv))\|_2\\
&\le C\|\nabla \vv \|_2 \|\nabla \vv \|_3 + C_2 \left( \int_{\R^3} (1+|\dd (\vv)|^2)^{\frac{r-2}{2}}|\dd (\nabla \vv)|^2 \; dx \right )^{\frac{1}{2}} + \gamma_0\|\nabla p\|_2
\end{align*}
where for the second inequality, we used standard interpolation and Sobolev embedding and also \eqref{A1}--\eqref{A2} (see \cite{MNR} for details). Thus, \eqref{gradp} easily follows.

Finally, assume that we have sequence $\vv^n$ satisfying \eqref{opice1}--\eqref{opice4}. Hence, we can find the corresponding sequence $p^n$ solving \eqref{pressureeq} and satisfying \eqref{cp1}--\eqref{cp2}, where $p_1^n, p_2^n$ are defined\footnote{The Bochner measurability over time follows from the fact that the mapping that assigns to fixed $\vv$ some $p$ solving \eqref{pressureeq} is continuous. We refer to \cite{MNR} for details.}  through \eqref{opicka1}, \eqref{opicka2}. First, we show a.e. convergence of $p_1^n$. We find $g^n$ solving for some fixed $s$, such that $1<s<\frac{q_1}{2}$
$$
\triangle g^n = |p^n_1 - p_1|^{s-2}(p^n_1 - p_1).
$$
Without loss of generality assume that $q_1<3$. Therefore we get $s<2$. Hence, using \eqref{2g} and \eqref{cp1} we get that
\begin{equation}
\int_0^T\|\na^2 g^n\|^a_a \; dt  \le C \qquad \textrm{ for all } a \in \left(\frac{1}{s-1}, \frac{q_1}{2(s-1)}\right]. \label{polip}
\end{equation}
Then, we find a sequence of nonnegative functions $f_k$ such that $f_k(x)=1$ in $B(0,k)$, $f_k(x)=0$ in $\R^3 \setminus B(0,2k)$, and that satisfies $|\nabla f|\le Ck^{-1}$ and $|\nabla^2 f|\le C k^{-2}$. Next, multiplying the equation for $p_1^n-p_1$, i.e.,
$$
-\triangle (p^n_1-p_1) = \div \div (\vv^n \otimes \vv^n - \vv \otimes \vv)
$$
by $g^n f_k$, integrating over $(0,T)\times \R^3$ and using integration by parts, we find that
$$
\int_0^T(p_1^n-p_1,\triangle (g^nf_k))\; dt = -\int_0^T(\vv^n\otimes \vv^n - \vv\otimes \vv, \na^2 (g^nf_k))\; dt.
$$
Using, \eqref{opice1} and  \eqref{opice2} and the fact that we
integrate over a compact domain, we see that the integral on
the right hand side tends to zero as $n$ tends to infinity,
provided that \eqref{polip} is satisfied for some
$a>(\frac{q_1}{2})'$. Since $s<\frac{q_1}{2}$ this is always
true. Next, using the identity $\triangle (g^nf_k) = f_k
\triangle g^n + g^n \triangle f_k + 2\nabla g^n \cdot \nabla
f_k$ and the definition of $g^n$ we deduce that
\begin{align*}
&\limsup_{n\to \infty} \int_0^T\int_{\R^3} |p^n_1-p_1|^s f_k \;
dx \; dt \\
&\qquad\le \limsup_{n\to \infty} C\int_0^T
\int_{B(0,2k)}|p_1^n-p_1|(|\nabla g^n| k^{-1}+ |g^n|k^{-2}) \;
dx \; dt.
\end{align*}
Setting $a:=\frac{q_1}{2(s-1)}$ we
deduce from the fact that $q_1<3$ that $a>3$ and since
$s<\frac{q_1}{2}$ we also obtain that $a'<\frac{q_1}{2}$.
Consequently, using \eqref{polip} we obtain $\int_0^T \|\nabla
g^n\|_{\infty}^a + \|g^n\|_{\infty}^a \; dt \le C$. Finally,
the H\"{o}lder inequality, \eqref{cp1} and the fact that
$q_1<3$ imply that
\begin{align*}
\int_0^T \int_{B(0,2k)}|p_1^n-p_1|(|\nabla g^n| k^{-1} +
|g^n|k^{-2}) \; dx \; dt\le C (k^{-1} + k^{-2}) k^{\frac{3}{a}}
\overset{a>3}{\underset{k\to\infty}{\to}} 0
\end{align*}
and using this it is easy to deduce that $p^n_1 \to p_1$ a.e. in $(0,T)\times \R^3$.

Next, we apply the similar procedure also for $p_2^n$. First we find for some fixed $s\in (2,r')$ for which $s<3$, the function $g^{mn}$ solving
$$
\triangle g^{mn} = |p^n_2-p^m_2|^{s-2}(p^n_2-p^m_2)f_k^{s-1},
$$
where $f_k$ is defined above. It is consequence of \eqref{2g} and \eqref{cp2} that
\begin{equation}
\int_0^T \|\nabla^2 g^{mn}\|_a^a \le C \quad \textrm{ for all } a\in \left[ \frac{2}{s-1}, \frac{r'}{s-1} \right]. \label{polip2}
\end{equation}
Next, multiplying the equation for $p^n_2-p^m_2$ by $f_k
g^{mn}$, integrating over $(0,T)\times \R^3$ and using the
similar procedure as above, we obtain
\begin{align*}
&\int_0^T\int_{\R^3}|p_2^n-p_2^m|^s f_k^s \; dx \; dt
\le \int_0^T (\bs(p^n,\dd(\vv^n))-\bs(p^m,\dd(\vv^m)), \na^2 (g^{mn}f_k)) \; dt \\
&\quad+ C\int_0^T \int_{B(0,2k)} |p_2^n-p_2^m|(|\nabla
g^{mn}|k^{-1} + |g^{mn}|k^{-2}) \; dx \; dt=:I_1^{mn} +
I_2^{mn}.
\end{align*}
To estimate $I_1^{mn}$ we use \eqref{A1}--\eqref{A2} (see \cite{MNR} for details) to observe
\begin{align*}
I_1^{mn}&= \int_0^T (\bs(p^n,\dd(\vv^n))-\bs(p^n_1 +p_2 ,\dd(\vv)),\na^2 (g^{mn}f_k))\; dt \\
&\quad + \int_0^T (\bs(p^n_1+p_2,\dd(\vv))-\bs(p,\dd(\vv)),\na^2 (g^{mn}f_k))\; dt\\
&\quad +\int_0^T (\bs(p,\dd(\vv))-\bs(p^m_1 + p_2,\dd(\vv)),\na^2 (g^{mn}f_k))\; dt\\
&\quad + \int_0^T(\bs(p^m_1 + p_2,\dd(\vv))-
\bs(p^m,\dd(\vv^m)),\na^2 (g^{mn}f_k))\; dt\\
&\le o(m,n) + \gamma_0\int_0^T(|p_2^n-p_2|+|p_2^m-p_2|,\na^2(g^{mn} f_k))\; dt,
\end{align*}
where
$$
\lim_{n,m \to \infty} o(n,m) =0.
$$
Thus, using \eqref{2g} and the definition of $g^{mn}$, we finally derive the estimate
\begin{align*}
\int_0^T \|(p^n_2-p^m_2)f_k\|_s^s\; dt  &\le o(m,n) +
\gamma_0\int_0^T \|(p^n_2-p_2)f_k\|_s^s+ \|(p^m_2-p_2)f_k\|_s^s
\; dt \\
&\qquad + CI_2^{mn}.
\end{align*}
Hence, applying for $I_2^{mn}$ the same procedure as above, taking limit w.r.t. $n$ and $m$, using weak lower semicontinuity of norm and the fact that $\gamma_0 < \frac12$, we finally conclude that $p_2^n \to p_2$ a.e. in $(0,T)\times \R^3$ which consequently implies \eqref{cp}.
\end{proof}


\section{The $\delta$ - approximation}


Throughout this section $\delta>0$ is
fixed real number. In order to take easily the limit from the
Galerkin (finite-dimensional) approximation to a continuous
(infinite-dimensional) approximation we have to define carefully the
approximation of \eqref{eq1}--\eqref{CD}. We introduce
the following $\delta$ - approximation:
\begin{equation}
\label{eq1-app} \vv_t+\div(\vv\otimes \vv)-\div(\bs(p,\dd(\vv)))+\delta\Delta^2 \vv=-\nabla p \quad\mbox{in}\quad[0,T]\times\R^3,
\end{equation}
\begin{equation}
\label{press-app}
\div \vv = 0 \quad\mbox{in}\quad[0,T]\times\R^3,
\end{equation}
\begin{equation}
\label{CD-app}
\vv(0,\cdot)=(\vv_0)_\delta\quad\mbox{in}\quad\R^3.
\end{equation}
Here the symbol $(\cdot)_\de$ stands for usual mollification by convolution.
The main existence result of this section is the following
\begin{prop}
\label{exis-galer}
Assume that $\vv_0\in L^2_{div}(\mathbb{R}^3)$ and $1\le r < 2$. Then there exist
\begin{align*}
\vv&:=\vv^{\delta}\in W^{1,2}(0,T; L^2(\mathbb{R}^3)^3)\cap L^{\infty}(0,T;W^{2,2}_{\div}(\mathbb{R}^3)^3),\\
\intertext{and}
p&:=p^{\delta}\in L^{2}(0,T;W^{1,2}(\mathbb{R}^3)),
\end{align*}
satisfying \eqref{eq1-app}-\eqref{press-app} weakly, i.e.,
\begin{equation}
\begin{split}
&\int_0^T (\vv_t,\bpsi)-(\vv \otimes \vv, \nabla
\bpsi)+(\bs(p,\dd(\vv)),\dd(\bpsi))\; dt\\
&+\delta\, \Int_0^T(\nabla^2\vv,\nabla^2\bpsi)\; dt=
\int_0^T\,(p,\div \bpsi)\; dt \quad \forall\;\bpsi\in
L^2(0,T;W^{2,2}(\mathbb{R}^3)^3).
\end{split}
\label{weak2}
\end{equation}
\end{prop}
\begin{proof}[Proof of Proposition \ref{exis-galer}]  We proceed via Faedo-Galerkin scheme. We look for  $\vv^N=\vv^N(t,x):=\Sum_{k=1}^N\,c_k^N(t)\,\bo^k(x)$ that is the
solution of Faedo-Galerkin approximation of
\eqref{eq1-app}-\eqref{press-app},i.e., we look for a vector $
\bc^N(t):=(c_1^N(t),c_2^N(t),\cdots,c_N^N(t))\in\R^N$ that
solves the following system of ordinary differential equations
\begin{equation}
\label{ODE} \left\{
\begin{aligned}
&\left(\frac{d
\vv^N}{dt},\bo^k\right)-\left(\vv^N_{\div}\otimes
\vv^N,\nabla\bo^k\right)+\left(\bs({\mathcal
P}(\vv^N),\dd(\vv^N)),\dd(\bo^k)\right)\\
&+\de\left(\nabla^2\vv^N,\nabla^2\bo^k\right)=
0,\quad
k=1,\cdots,N,\\&\vv^N(0)=(\vv^N_0)_\delta,
\end{aligned}
\right.
\end{equation}
where $(\vv^N_0)_\delta:=\Sum_{k=1}^N\;\langle(v_0)_\delta,\bo^K\rangle_{W^{2,2}}\,\bo^k$ and ${\mathcal
P}(\vv^N):=p^N$, where $p^N$ is given by \eqref{pressureeq} with $\vv^N$.
The initial value problem \eqref{ODE} can be written in the form
\begin{equation}
\label{ODE1} \dot{\mbox{\sf\large c}}^N={\mathcal
F}\left(\mbox{\sf\large c}^N\right)\quad\mbox{with}\quad
\mbox{\sf\large c}^N(0)=\mbox{\sf\large c}_0^N,
\end{equation}
where ${\mathcal F} : \R^N\to\R^N$ is a continuous function. Hence we obtain an ordinary differential equation for which the Caratheodory theory can be applied. In conclusion, for every $N\geq 1$ there exists an approximate solution $\vv^N$ on the short time interval $[0,\tau]$ that can be however extended onto the whole interval of interest $(0,T)$ by using the estimates proved below.  Our aim is to pass to the limit as $N\to\infty$ (for fixed $\delta$). To do it, we shall first establish a priori estimates that are independent of $N$.

Next , multiplying the k-th equation in \eqref{ODE}$_1$ by $c^N_k(t)$, using the fact that convective term vanishes after integration by parts ($\vv^N$ has compact support),  using the estimate \eqref{visc1} and integrating the result over times interval $(0,t)$, we deduce that
\beq
\label{ener1}
\begin{split}
\frac12 \|\vv^N(t)\|_2^2+C_1 \int_0^t \int_{\mathbb{R}^3} \phi(|\dd (\vv^N)|) \; dx \; dt + \int_0^t\delta \|\nabla^2 \vv^N\|_2^2 \le \frac12\|(\vv_0^N)_\delta\|_2^2.
\end{split}
\eeq
Now, using Proposition \ref{Ppres}, \eqref{ener1} and the fact that $\vv_0\in L^2(\mathbb{R}^3)^3$, we deduce the
inequality\footnote{Note that we use the fact that
$\|\na\vv^N\|_2^2 = (\nabla \vv^N, \nabla \vv^N)=-(\triangle \vv^N, \vv^N) \le \|\nabla^2 \vv^N\|_2^2 + \|\vv^N\|_2^2$.}
\begin{equation}
\begin{split}
\label{ener2}
\sup_{t\in (0,T)}\|\vv^N(t)\|_2^2+2 C_1&\Int_0^T\Int_{\R^3}\phi(|\dd(\vv^N)|) \; dx\; dt \\ &\qquad+\Int_0^T\,\Big(\de\| \vv^N\|_{W^{2,2}}^2 +C(\delta)
\|p^N\|_{1,2}^2\Big) \; dt\le C.
\end{split}
\end{equation}
It is worth of noticing that the constant appearing on the right-hand side of \eqref{ener2} is $\delta$-independent. Next, we deduce the similar estimate also for time derivative $\vv_t$. Hence, multiplying the $k$-th equation in \eqref{ODE}$_1$ by $(c^N_K(t))_t$ and using integration by parts ($\vv^N$ has compact support), we deduce that
\beq
\label{ener3}
\|\vv^N_t(t)\|_2^2+\frac{\de}{2}\frac{d}{dt}\|\na^2\vv^N(t)\|_2^2=\cI_1+\cI_2,
\eeq
where
\beqn
\cI_1&=&-\left(\div\left(\vv^N\otimes\vv^N\right),\vv^N_t \right),\\
\cI_2&=&\left(\div\left(\bs(p^N,\dd(\vv^N))\right),\vv^N_t\right).
\eeqn
Using the H\"older inequality and \eqref{helm2}, we deduce
\beqn
\cI_1&\leq&\|\div (\vv^N\otimes\vv^N)\|_2\|\vv^N_t\|_2 \leq C\|\vv^N\|_3\|\na\vv^N\|_6\|\vv^N_t\|_2\\
&\leq&\frac{1}{6}\|\vv^N_t\|_2^2+C\|\na \vv^N\|_2 \|\vv^N\|_2\|\na^2\vv^N\|_2^2.
\eeqn
Next, using \eqref{A1}--\eqref{A2}, we conclude
\beqn
\cI_2\leq\|\div\left(\bs(p^N,\dd(\vv^N))\right)\|_2\,\|\vv^N_t\|_2\leq \frac{1}{6}\|\vv^N_t\|_2^2+C\left(\|\na p^N\|_2^2+\|\vv^N\|_{2,2}^2\right).
\eeqn
Finally, substituting the estimates for $\mathcal{I}_1$ and $\mathcal{I}_2$ into  \eqref{ener3}, we obtain
\begin{equation}
\label{ener3.1}
\|\vv^N_t\|_2^2 + C\delta \frac{d}{dt} \|\nabla^2 \vv^N\|_2^2 \le C(1+\|\na p^N\|_2^2+ \|\na \vv^N\|_2 \|\vv^N\|_2)(1+\|\nabla^2 \vv^N\|_2^2).
\end{equation}
Consequently, using \eqref{ener2} and the Gronwall inequality we conclude
\beq
\label{ener4}
\sup_{t\in (0,T)}\|\na^2\vv^N(t)\|_2^2 + \Int_0^T\,\|\vv^N_t\|_2^2\;dt\le C(\delta).
\eeq

Having \eqref{ener2} and \eqref{ener4} and using \eqref{visc2}, we see that there is $(\vv,p,\overline{\bs})$ such that (up to subsequence)
\beqn
\vv^N\rightharpoonup^* \vv&\quad\mbox{weakly$^*$ in}\quad& L^{\infty}(0,T;W^{2,2}(\R^3)^3),\\
\vv^N_t\rightharpoonup \vv_t&\quad\mbox{weakly in}\quad& L^2(0,T; L^2(\R^3)^3),\\
p^N\rightharpoonup^* p&\quad\mbox{weakly$^*$ in}\quad& L^{\infty}(0,T; W^{1,2}(\R^3)),\\
\bs(p^N,\dd(\vv^N))\rightharpoonup^* \overline{\bs} &\quad\mbox{weakly$^*$ in}\quad& L^{\infty}(0,T;L^2(\R^3)^{3\times 3}).
\eeqn
It remains to pass to the limit in the equation \eqref{ODE}.
As usually, the most delicate terms to deal with are the nonlinear one. The classical tool is to use some compactness arguments in order
to obtain strong convergence in the adequate function spaces. Hence, by using Aubin-Lions compactness lemma, we get (up to subsequence)
\beqn
\vv^N\rightarrow\vv&\quad\mbox{strongly in}\quad&L^2(0,T;W^{1,2}_{loc}(\R^3)^3),\\
\vv^N\rightarrow\vv&\quad\mbox{a.e. in}\quad&[0,T]\times\R^3,\\
\nabla \vv^N\rightarrow \nabla \vv&\quad\mbox{a.e. in}\quad&[0,T]\times\R^3.
\eeqn
Consequently, using Proposition \ref{Ppres}, we deduce
\beqn
p^N\rightarrow p&\quad\mbox{a.e. in}\quad&[0,T]\times\R^3.
\eeqn
Using these convergence results it is easy to deduce that $\overline{\bs}=\bs(p,\dd (\vv))$ and to set $N\to\I$ in \eqref{ODE} and obtain a solution $(\vv, p)$ to
the equations \eqref{pressureeq} and  \eqref{weak2}, where \eqref{weak2} holds for all $\bpsi$ with $\div \bpsi=0$. Finally using the Helmolhtz decomposition and the equation for the pressure \eqref{pressureeq} we finally obtain the validity of \eqref{weak2} for all given $\bpsi$. The proof of Proposition \ref{exis-galer} is complete.
\end{proof}


\section{Limit $\delta \to 0$}


In this subsection we denote $(\vv^{\de}, p^{\de})$ a solution from Proposition \ref{exis-galer}.  Our main goal in this section is to set  $\de \to 0$ in the equation \eqref{eq1-app} and to get \eqref{weak21}.

First, using weak lower semicontinuity of norms and the Fatou lemma, it is easy to deduce from \eqref{ener2} and from the estimate \eqref{INT} that
\begin{equation}
\begin{split}
\label{ener-bound}
\sup_{0\leq t\leq T}\|\vv^\de (t)\|_2^2 + 2C_1\Int_0^T&\Int_{\R^3}\Big(\phi(|\dd(\vv^\de)|) + |\vv^\de|^q\; dx\Big) \; dt \\
&\qquad+ \int_0^T\Big(\delta \|\nabla^2 \vv^\de \|_2^2 + C(\delta)\|p^\de \|_{1,2}^2\Big) \; dt \le C,
\end{split}
\end{equation}
for all $q\in [2,\frac{5r}{3})$.
To obtain uniform estimate also on the pressure, we use Proposition \ref{Ppres} to obtain the decomposition $p^{\de}=p_1^\de + p_2^\de$ such that for all $q\in (2,\frac{5r}{6})$ and all $s\in [2,r']$ there holds
\begin{equation}
\int_0^T \|p_1^\de \|_q^q + \|p_2^\de\|_s^s \le C(q,s). \label{ener-bound2}
\end{equation}
Similarly, \eqref{ener4} implies that
\beq
\label{ener4.1}
\sup_{t\in (0,T)}\|\na^2\vv^{\de}(t)\|_2^2 + \Int_0^T\,\|\vv^{\de}_t\|_2^2\;dt  \le C(\delta).
\eeq

Therefore, using \eqref{ener-bound}, \eqref{ener-bound2} and the Aubin-Lions lemma, we can extract a not relabeled subsequence such that
\begin{align}
\vv^{\de} &\rightharpoonup^* \vv &&\textrm{ weakly$^*$ in } L^{\infty}(0,T; L^2(\R^3)^3),\label{pit1}\\
\vv^{\de} &\rightharpoonup \vv &&\textrm{ weakly in } L^{r}(0,T; W^{1,r}_{loc}((\R^3)^3),\label{pit2}\\
\vv^{\de}_t &\rightharpoonup \vv_t &&\textrm{ weakly in } L^{\frac{5r}{6}}(0,T; W^{-2,2}((\R^3)^3),\label{pit2.1}\\
\vv^{\de} &\rightarrow \vv &&\textrm{ a.e. in } (0,T)\times \R^3,\label{pit2.1.1}\\
p^\de_1 &\rightharpoonup p_1 &&\textrm{ weakly in } L^{q}(0,T; L^q(\R^3)) \textrm{ for all } q\in (1,\frac{5r}{6}],\label{pit3}\\
p^\de_2 &\rightharpoonup p_2 &&\textrm{ weakly in } L^{s}(0,T; L^s(\R^3)) \textrm{ for all } s\in [2,r'].\label{pit4}
\end{align}
These, estimates are sufficient to pass to the limit from \eqref{weak2} to \eqref{weak21} if we show the point-wise convergence also for $\na \vv^{\de}$ and $p^\de$.

To do it, we first derive some regularity estimates.
For this purpose we observe first that $\vv^\de$ is more regular. Indeed, $\vv^\de$ solves
\beq
\label{regular11}
\Delta^2\vv^\de=F,
\eeq
where (after using \eqref{ener-bound} and \eqref{ener4.1})
$$
F:=\div \bs(p^\de, \dd (\vv^\de)) -\div\left(\vv^{\de}\otimes\vv^\de \right)-\na p^\de- \vv{^\de}_t \in L^2(0,T; L^2(\R^3)^3).
$$
Thus, using \eqref{ener2}, we get (for fixed $\de$)
$$
\vv^{\de}\in L^2(0,T;W^{4,2}(\R^3)^3).
$$
Hence, we see that we can set  $\bpsi:=-\triangle \vv^{\de}$ in \eqref{weak2}.  As the result, we obtain (after integration by parts and also using \eqref{A1})
\beq
\label{test-Lap1}
\frac{1}{2}\frac{d}{dt}\|\na\vv^\de (t)\|_2^2+\de\|\na^3\vv^\de (t)\|_2^2+C_1\II_r(\vv^\de)\leq\cJ_1+\cJ_2,
\eeq
where
\begin{align*}
\II_r(\vv^\de):&=\Int_{\R^3}\,\left(1+|\dd(\vv^\de)|^2\right)^{\frac{r-2}{2}}\,|\dd(\na\vv^\de)|^2\,dx,\\
\cJ_1:&=\left(\frac{\partial\bs(p^\de,\dd(\vv^\de))}{\partial p}\na p^\de,\Delta\vv^\de\right),\\
\cJ_2:&=\left(
\div\left(\vv^{\de}\otimes\vv^\de\right),\Delta\vv^\de\right).
\end{align*}
For $\cJ_2$ we use the fact that
$\div\vv^{\de}=0$  to obtain
$$
\cJ_2=-\Sum_{i,j,k}\,\Int_{\R^3}\,\partial_k(\vv^{\de})^i\partial_i(\vv^{\de})^j\partial_k(\vv^{\de})^j\,dx
\leq C\,\|\na \vv^{\de}\|_3^3.
$$
Using the assumption \eqref{A2}, we get
\begin{align*}
\cJ_1&\leq \gamma_0\Int_{\R^3}\,\left(1+|\dd(\vv^{\de})|^2\right)^{\frac{r-2}{4}}|\dd(\na\vv^{\de})||\na p^{\de}|\,dx,\\
&\leq \gamma_0\left(\II_r(\vv^{\de})\right)^{\frac12}\|\na p^{\de}\|_2.
\end{align*}
Therefore, substituting the relations for $\cJ_1$-$\cJ_2$ into \eqref{test-Lap1}, we get
\begin{equation}
\begin{split}
\label{test-Lap1.1} \frac{1}{2}\frac{d}{dt}\|\na\vv^\de
(t)&\|_2^2\!+\de \|\na^3\vv^\de (t)\|_2^2\!+C_1\II_r(\vv^\de) \le
C\|\na \vv^{\de}\|_3^3+ \gamma_0
\left(\II_r(\vv^\de)\right)^{1/2}\|\nabla p^{\de}\|_2.
\end{split}
\end{equation}

To bound the right hand side of \eqref{test-Lap1.1}, we
use the estimate \eqref{gradp}. Since, $\gamma_0<\frac{C_1}{C_1+C_2}$,
we deduce that $\frac{\gamma_0 C_2}{1-\gamma_0}<C_1$ and
therefore substituting \eqref{gradp} into \eqref{test-Lap1.1} and using the Young and Korn inequalities, we finally deduce that (see \cite{MNR} for details)

\begin{equation}
\begin{split}
\label{regular}
\frac{1}{2}\,\frac{d}{dt}\|\na\vv^{\de}\|_2^2+C\II_r(\vv^{\de})\leq C\left(\|\dd(\vv^\de)\|_3^3 + \|\dd(\vv^\de)\|_3^2\|\nabla \vv^\de\|^2_2\right).
\end{split}
\end{equation}

Having the estimate \eqref{regular} at hand we proceed in a similar way as in \cite{MNR} but now we need to be more careful  since the behavior at infinity has to be discussed.
\subsection{Step 1}
First,we recall two interpolation inequalities
\begin{align}
\|\cdot\|_3 &\le \|\cdot\|_2^{\alpha}\|\cdot\|_{3r}^{1-\alpha}\qquad &&\alpha=\frac{2(r-1)}{3r-2},\label{In1}\\
\|\cdot\|_3 &\le \|\cdot\|_r^{\beta}\|\cdot\|_{3r}^{1-\beta} \qquad &&\beta=\frac{r-1}{2}\,. \label{In2}
\end{align}
In what follows we also use the following notation
$$
\|v\|_{s,\le}^s := \int_{\{|v|\le 1\}} |v|^s \; dx, \qquad \|v\|_{s,\ge}^s := \int_{\{|v|> 1\}} |v|^s \; dx
$$
The reason why we interpolate in \eqref{In1}--\eqref{In2} into $L^{3r}$ norm is the following lemma
\begin{lem}
The following estimate holds
\begin{equation}
\II_r(\vv^\de) \ge C \|\dd(\vv^\de)\|_{3r,\ge}^r. \label{Base}
\end{equation}
\end{lem}
\begin{proof}
Since
$$
|\nabla \left( (1+|\dd(\vv^\de)|^2)^{\frac{r}{4}}-1\right)|^2\le C(1+|\dd(\vv^\de)|^2)^{\frac{r-2}{2}}|\dd(\nabla \vv^\de)|^2
$$
we obtain from standard embedding and from the simple inequality $(1+x)^{\frac{r}{4}}-1\ge Cx^{\frac{r}{4}}$, that is  valid for some $C>0$ and all $x\ge 1$, that
$$
\||\dd(\vv^\de)|^{\frac{r}{2}}\|_{6,\ge}^2\le C \|(1+|\dd(\vv^\de)|^2)^{\frac{r}{4}}-1\|_6^2 \le C \II_r(\vv^\de).
$$
And \eqref{Base} immediately follows.
\end{proof}
To estimate the second term in the right hand side of \eqref{regular} we  compute
\begin{align*}
\|\dd(&\vv^\de)\|_3^2\|\nabla \vv^\de\|_2^2 \le \|\dd(\vv^\de)\|_{3,\le}^2\|\nabla \vv^\de\|_2^2+\|\dd(\vv^\de)\|_{3,\ge}^2\|\nabla \vv^\de\|_2^2\\
&\le C(\|\dd(\vv^\de)\|_{2,\le}^2+1)\|\nabla \vv^\de\|_2^2 + \|\dd(\vv^\de)\|_{3,\ge}^{2}\|\nabla \vv^\de\|_2^2\\
&\overset{\eqref{In1},\eqref{In2}}\le C(\|\dd(\vv^\de)\|_{2,\le}^2+1)\|\nabla \vv^\de\|_2^2 + \|\dd(\vv^\de)\|_{3r,\ge}^{2(1-\beta)}\|\dd(\vv^\de)\|_{r,\ge}^{2\beta}\|\nabla \vv^\de\|_2^{2}\\
&\overset{\eqref{Base}}\le
C(\|\dd(\vv^\de)\|_{2,\le}^2+1)\|\nabla \vv^\de\|_2^2 +
\left(\II_r(\vv^\de)\right)^{\frac{2(1-\beta)}{r}}\|\dd(\vv^\de)\|_{r,\ge}^{2\beta}\|\nabla
\vv^\de\|_2^{2}.
\end{align*}
Applying the Young inequality, we conclude that
\begin{align*}
\|\dd(\vv^\de)\|_3^2\|\nabla \vv^\de\|_2^2 &\le
C(\|\dd(\vv^\de)\|_{2,\le}^2+1)\|\nabla \vv^\de\|_2^2 + \delta
\II_r(\vv^\de)\\
&\qquad  + C\|\dd(\vv^\de)\|_{r,\ge}^{r+{\frac {\left
(2-r\right )r}{2r-3}}}\|\nabla \vv^\de\|_2^{\frac{2r}{2r-3}}.
\end{align*}
Finally, using the fact that $\|\dd (\vv^\de)\|_{r,\ge}\le \|\nabla \vv^\de\|_2^{\frac{2}{r}}$ we conclude
\begin{equation}
\begin{split}
\|\dd(\vv^\de)\|_3^2\|\nabla \vv^\de\|_2^2 &\le
C(\|\dd(\vv^\de)\|_{2,\le}^2+1)\|\nabla \vv^\de\|_2^2 + \delta
\II_r(\vv^\de) \\
&\quad + C\|\dd(\vv^\de)\|_{r,\ge}^{r}\|\nabla
\vv^\de\|_2^{\frac{4}{2r-3}}.\label{MB2}
\end{split}
\end{equation}

Next, for fixed $\mu\in(0,1)$ we estimate the first term on the right hand side of \eqref{regular} similarly. Hence
\begin{align*}
\|\dd(&\vv^\de)\|_3^3\le C\|\nabla \vv^\de\|_2^2 + \|\dd(\vv^\de)\|_{3,\ge}^{3\mu}\|\dd(\vv^\de)\|_{3,\ge}^{3(1-\mu)}\\
&\overset{\eqref{In1},\eqref{In2}}\le \|\nabla \vv^\de\|_2^2 + \|\dd(\vv^\de)\|_{3r,\ge}^{3\mu(1-\alpha)+3(1-\mu)(1-\beta)}\|\dd(\vv^\de)\|_{r,\ge}^{3(1-\mu)\beta}\|\nabla \vv^\de\|_2^{3\mu\alpha}\\
&\le \|\nabla \vv^\de\|_2^2 + \left(\II_r(\vv^\de)\right)^{\frac{3\mu(1-\alpha)+3(1-\mu)(1-\beta)}{r}}\|\dd(\vv^\de)\|_{r,\ge}^{3(1-\mu)\beta}\|\nabla \vv^\de\|_2^{3\mu\alpha}.
\end{align*}
Finally, applying the Young inequality we observe that
\begin{align}
\|\dd(\vv^\de)\|_3^3 &\le \|\nabla \vv^\de\|_2^2 + \delta
\II_r(\vv^\de) + C\|\dd(\vv^\de)\|_{r,\ge}^A \|\nabla
\vv^\de\|_2^B \label{MB3}
\end{align}
with $A,B$ given as
\begin{equation}
\begin{split}
A&:=3{\frac {\left (r-1\right )\left (-1+\mu\right )\left (3\,r-2\right
)r}{-15\,{r}^{2}+37\,r-27\,\mu\,r-18+18\,\mu+9\,\mu\,{r}^{2}}},\\
B&:=-12{\frac {\mu r\left (r-1\right )}{-15{r}^{2}+37r-27\,\mu r-
18+18 \mu+9\mu {r}^{2}}}\,.
\end{split}\label{MB22}
\end{equation}
Next, setting (note that for $r\in (\frac95, 2)$, $\mu\in (0,1)$)
\begin{equation}
\mu:=-{\frac {3{r}^{2}-11r+6}{6(r-1)}}, \label{mu}
\end{equation}
we observe that $A=r$ and
\begin{equation}
B=\frac{4(3-r)}{3r-5}. \label{B}
\end{equation}
Since $B\ge \frac{4}{2r-3}$ for all $r\in(1,2)$ we obtain by combining \eqref{MB2}, \eqref{MB3} and \eqref{regular} that
\begin{equation}
\frac{d}{dt} \|\nabla \vv^\de\|_2^2 + C \II_r(\vv^\de) \le C(1+\|\dd(\vv^\de)\|_{r,\ge}^r + \|\dd(\vv^\de)\|_{2,\le}^2) (1+\|\nabla \vv^\de\|_2^2)^\lambda \label{MB4}
\end{equation}
with
\begin{equation}
\lambda:=\frac{2(3-r)}{3r-5}. \label{lam}
\end{equation}
Finally, dividing \eqref{MB4} by $(1+\|\nabla \vv^\de\|_2^2)^{\lambda}$, integrating w.r.t. time and using the estimate \eqref{ener-bound}, we finally deduce (see \cite{MNRR} for details)
\begin{equation}
\int_0^T \II_r(\vv^\de)(1+\|\nabla \vv^\de\|_2^2)^{-\lambda} \; dt\le C\,. \label{MB6}
\end{equation}
\subsection{Step 2} Here, we deduce how \eqref{MB6} implies local compactness of $\dd(\vv^\de)$.

\begin{lem}
There is  $\gamma\in (0,1)$ such that
\begin{equation}
\int_0^T \left(\II_r(\vv^\de)\right)^{\gamma} \; dt \le C. \label{comp}
\end{equation}
\end{lem}
\begin{proof}
We can compute
\begin{align*}
\left(\II_r(\vv^\de)\right)^{\gamma} &=\left(\II_r(\vv^\de)\right)^{\gamma}(1+\|\nabla \vv^\de\|_2^2)^{-\gamma\lambda}(1+\|\nabla \vv^\de\|_2^2)^{\gamma\lambda}\\
&\le C \left(\II_r(\vv^\de)\right)^{\gamma}(1+\|\nabla \vv^\de\|_2^2)^{-\gamma\lambda}(1+\|\dd(\vv^\de)\|_{2,\le}^2 + \|\dd( \vv^\de)\|_{2,\ge}^2)^{\gamma\lambda}\\
\intertext{and by using interpolation inequality $\|\cdot\|_2\le \|\cdot\|_r^{\frac{3r-2}{4}}\|\cdot\|_{3r}^{\frac{3(2-r)}{4}}$ and \eqref{Base} we deduce that}
\left(\II_r(\vv^\de)\right)^{\gamma}&\le
C \left(\II_r(\vv^\de)\right)^{\gamma}(1+\|\nabla \vv^\de\|_2^2)^{-\gamma\lambda}(1+\|\dd(\vv^\de)\|_{2,\le}^2
\\&\qquad+ \|\dd( \vv^\de)\|_{r,\ge}^{\frac{3r-2}{2}} \left(\II_r(\vv^\de)\right)^{\frac{3(2-r)}{2}})^{\gamma\lambda}\,.\\
\intertext{Using the Young inequality with coefficients
$\frac{1}{\gamma}$ and $\frac{1}{1-\gamma}$, we continue
as}\left(\II_r(\vv^\de)\right)^{\gamma} &\le C
\frac{\II_r(\vv^\de)}{(1+\|\nabla \vv^\de\|_2^2)^{\lambda}}+
C(1+\|\dd(\vv^\de)\|_{2,\le}^{\frac{2\gamma\lambda}{1-\gamma}})
\\
&\qquad+ C\|\dd(
\vv^\de)\|_{r,\ge}^{\frac{3r-2}{2}\frac{\gamma\lambda}{1-\gamma}}
\left(\II_r(\vv^\de)\right)^{\frac{3(2-r)}{2r}\frac{\gamma\lambda}{1-\gamma}}\,.
\end{align*}
Finally, applying once again the Young inequality onto the last term and integrating w.r.t. time we see that we can choose $\gamma$ so small that by using \eqref{MB6} and \eqref{ener-bound} we obtain \eqref{comp}.
\end{proof}
The next lemma finally gives the desired compactness of the velocity gradient.
\begin{lem}
The following estimate holds
\begin{equation}
\int_0^T \left(\int_{B(0,R)}|\nabla^2 \vv^\de|^r\; dx \right)^{\gamma}\; dt \le C(R)\,.
\label{MB8}
\end{equation}
\end{lem}
\begin{proof}
First, we can deduce that
\begin{align*}
&\int_{B(0,R)}|\dd (\nabla \vv^\de)|^r\; dx \\
&\qquad \le \int_{B(0,R)}\left(
(1+|\dd(\vv^\de)|^2)^{\frac{r-2}{2}}|\dd(\nabla \vv^\de)|^2
\right)^{\frac{r}{2}}(1+|\dd(\vv^\de)|^2)^{\frac{r(2-r)}{4}}\;
dx\,.
\end{align*}
Using the Young inequality, we deduce
\begin{align*}
\int_{B(0,R)}|\dd (\nabla \vv^\de)|^r\; dx &\le \int_{B(0,R)}
(1+|\dd(\vv^\de)|^2)^{\frac{r-2}{2}}|\dd(\nabla \vv^\de)|^2 \;
dx \\
&\qquad+ C\int_{B(0,R)} \big(1+|\dd(\vv^\de)|^r\big)\; dx.
\end{align*}
Finally, using $\gamma$-power, integrating w.r.t. time and using \eqref{comp}, \eqref{ener-bound} and the fact that we integrate over the domain with finite measure, we obtain \eqref{MB8}.
\end{proof}

Finally, using the standard interpolation, \eqref{ener-bound}, \eqref{MB8} and the Young inequality, we conclude
$$
\int_0^T\|\vv^\de\|_{1+\sigma,r;B(0,R)}^r\le \int_0^T\|\vv^\de\|_{1,r;B(0,R)}^{r(1-\gamma)}\|\vv^\de\|_{2,r;B(0,R)}^{r\gamma} \le C,
$$
provided that $\sigma \ll 1$ is sufficiently small. Using the compact embedding $W^{1+\sigma,r}$ into $W^{1,r}$ we then conclude (up to subsequence) that
$$
\dd(\vv^\de) \to \dd(\vv) \quad \textrm{a.e. in } (0,T)\times\mathbb{R}^3.
$$
We skip the details here and we refer to \cite{MNRR, MNR} for a complete presentation. Consequently, we may use Proposition \ref{Ppres} and to finally deduce that
$$
p^\de \to p \quad \textrm{a.e. in } (0,T)\times\mathbb{R}^3,
$$
which finishes the proof of Theorem \ref{MAIN}.



\end{document}